\documentclass[12pt]{amsart}
\usepackage[utf8]{inputenc}
\usepackage{amsmath,amssymb,amsfonts}
\usepackage{color} 
\usepackage{xcolor}
\usepackage{hyperref} 
\usepackage{epstopdf}
\usepackage{enumitem}

\newcommand*\circled[1]{\tikz[baseline=(char.base)]{
            \node[shape=circle,draw,inner sep=1pt,thick] (char) {#1};}}

\usepackage{amsbsy}
\usepackage{amsthm}
\usepackage{mathrsfs}  
\usepackage{xfrac}
\usepackage[BCOR=0.5cm]{typearea}
\usepackage{tikz-cd}
\usetikzlibrary{decorations.pathmorphing}

\numberwithin{equation}{section}

\theoremstyle{plain}
\newtheorem{theorem}{Theorem}[section]
\newtheorem*{theorem*}{Theorem}

\newtheorem{proposition}[theorem]{Proposition}

\newtheorem{claim}[theorem]{Claim}
\newtheorem{example}[theorem]{Example}
\newtheorem{definition}[theorem]{Definition}
\numberwithin{theorem}{section}

\DeclareMathOperator{\PFZ}{PFZ}
\DeclareMathOperator{\PF}{PF}
\newcommand{\pfz}[1]{\psi(#1)}
\def\decomp{\mathcal P}
\def\Z{\mathbb Z}
\DeclareMathOperator{\diag}{diag}
\DeclareMathOperator{\secdinv}{\dinv_2}
\DeclareMathOperator{\pridinv}{\dinv_1}
\DeclareMathOperator{\dinv}{dinv}
\newcommand{\mydef}[1]{{\bf #1}}
\DeclareMathOperator{\row}{row}
\DeclareMathOperator{\dom}{dom}
\DeclareMathOperator{\col}{col}

\newcommand{\rowa}[2]{\row_{#1}(#2)}
\newcommand{\cola}[2]{\col_{#1}(#2)}
\DeclareMathOperator{\HT}{ht}
\DeclareMathOperator{\WD}{wd}
\newcommand{\hta}[2]{\HT(#1_{#2})}
\newcommand{\diaga}[2]{\diag_{#1}(#2)}
\def\mainmap{\Psi}

\def\numocd{\phi}
\newcommand{\maxdiag}[1]{\Delta_{#1}}

\def\mcc{\mathfrak c}
\def\mcd{\mathfrak d}

\usepackage{todonotes}

\newcommand{\omitt}[1]{}

\author{Susanna Fishel}
\address{School of Mathematical and Statistical Sciences, Arizona State University, P.O. Box 871804, Tempe, AZ 85287-1804, USA} \email{sfishel1@asu.edu}

\author{Luis Pena}
\address{Instituto de Matem\'atica y F\'{\i}sica, Universidad de Talca, Casilla 747, Talca,
Chile} \email{luis.cardenas@utalca.cl}

\title{Parking functions with zero secondary dinv}
\date{\today}

\begin{document}

\thanks{Funding: The first author was partially supported by
 Simons Collaboration Grants for
Mathematicians \#359602 and \#709671. Work began when the first author was at University of Talca on a Fulbright Scholar Grant. The second author was partially supported by a Beca de Doctorado Nacional ANID, folio 21211108.}

\dedicatory{Dedicated to Mourad Ismail.}
\begin{abstract}
In this note, we show that the number of parking functions of length $n$ with zero secondary dinv is equal to the number of ordered cycle decompositions of permutations of $[n]$.
  \end{abstract}
\maketitle
\section{Introduction}

Carlsson and Mellit proved the renowned shuffle conjecture in \cite{CM}; that is they proved that the conjectured combinatorial formula for the Frobenius character of the diagonal coinvariant algebra is correct. The combinatorial expression is in terms of parking functions, using the statistics area and dinv.

The number of diagonal inversions or dinv is a well-studied statistic, first defined for Dyck paths, then expanded to all parking functions. Please see the standard reference \cite{haglund2008}. In \cite{gxz2014} Garsia, Xin, and Zabrocki defined  primary and secondary dinv, which refine dinv. In their proof of the three shuffle case of a refinement of the shuffle conjecture, they found and used a bijection and noticed that their bijection swapped primary and secondary dinv. We enumerate the parking functions with zero secondary dinv.

 Let $\PFZ(n)$ denote the set of parking functions defined on $[n]$ with zero secondary dinv and $\pfz n := | \PFZ (n) |$, where we set $\pfz 0=1$.  We will show that for $n \geq 1$, $\pfz n$ satisfies the following recursion
\begin{equation}\label{Recursion}
\pfz n= \sum_{k=1}^{n} \dfrac{(n-1)!}{ (n-k)! } \sum_{\ell=0}^{n-k} \binom{n-k}{\ell}  \pfz \ell  \pfz {n-k-\ell} .
\end{equation}

We will prove \eqref{Recursion} by proving that any parking function $h \in \PFZ (n)$ can be decomposed uniquely into a triple $(f, g, D)$, where $f$ and $g$ are two parking functions which have zero secondary dinv and whose domains are disjoint subsets of $[n]$ not containing $1$, and $D$ is a sequence consisting of the remaining elements of $[n]$ and beginning with $1$. This will give us a bijection.
Recursion \eqref{Recursion} allows us to enumerate $\PFZ(n)$. This recursion is satisfied by other combinatorial objects and defines the sequence A007840 on OEIS. We describe one of the objects it counts. 

Let $\numocd(n)$ be the number of ordered cycle decompositions of $n$, defined in Section~\ref{subsec:ocd}. Please also see \cite[Page 8]{AA}. %

 Our main result is
\begin{theorem}\label{thm:enumPFZ}Let $n$ be a nonnegative integer. We have $\pfz n=\numocd(n)$. \end{theorem}

As we want to prove that $\pfz n$ is the same as $\numocd(n)$, we
prove $\numocd (n)$ satisfies \eqref{Recursion}, with $\pfz{n}$ replaced
by $\numocd(n) $ throughout.

\begin{proposition}\label{prop:ocd_rec}
The recursion \eqref{Recursion} is satisfied by $\numocd (n)$, the number of ordered cycle decompositions of $n$.
\end{proposition}

We begin with the definitions of parking functions and dinv in
Section~\ref{sec:prelims}.  In Section~\ref{sec:map}, we begin our
proof of \eqref{Recursion} by defining a set $\decomp(n)$ of triples in \eqref{def:decomp} and \eqref{eq:decomp}. We define a process to insert an integer
into a parking function in Section~\ref{subsec:insertion}. In
Section~\ref{subsec:map}, we construct a map $\mainmap$ from the set $\decomp(n)$
to $\PFZ(n)$. Given a triple $(f,g,D)\in \decomp(n)$, where $f$ and $g$ are parking functions and $D$ is a sequence, we will insert
the elements of $D$ and $g$ into $f$ to produce
$h\in\PFZ(n)$. Not until
Section~\ref{sec:mapinv} do we show that $\mainmap$ is
invertible. Finally, in Section~\ref{subsec:enum}, we say a few words
on why the bijection proves \eqref{Recursion} and therefore
Theorem~\ref{thm:enumPFZ}.

\section{Preliminaries}
\label{sec:prelims}
\subsection{Parking functions}
There are many equivalent definitions of parking functions. We use
their diagrams to define them.
\begin{definition}
  \label{def:pfs}
  A \mydef{Dyck path} of length $2n$ is a lattice path from $(0,0)$ to $(n,n)$ which never goes below the line $y=x$. Let $A\subset\Z$, $|A|=n$. Starting with a Dyck path $D$ of length $2n$, a \mydef{parking function} on $A$ is an arrangement of the elements of $A$ in the squares immediately to the right of the $n$ vertical steps of $D$ and with strict decrease down the columns. We will often refer to the elements of $A$ as the \mydef{labels} of the parking function. 

  \end{definition}
The standard definition of a parking function is the case  $A=[n]$.

We can also view a parking function as a function $f:A\to[n]$, as the name
suggests. Given a parking function in diagram form, set $f(i)=j$ if
$i$ has been placed in column $j$. The domain of $f$ is $A$ and denoted
$\dom(f)$. We denote the set of parking functions on $A$ by $\PF(A)$
and in the case $A=[ n ]$ by $\PF(n)$. In this paper, we use both the
diagram form and the function form of a parking function.  Note that the function form is usually written as a vector $(f(1),f(2),\ldots,f(n))$. Please see
\cite{haglund2008} for more details and Figure~\ref{fig:pfAndDiagram}
for an example.

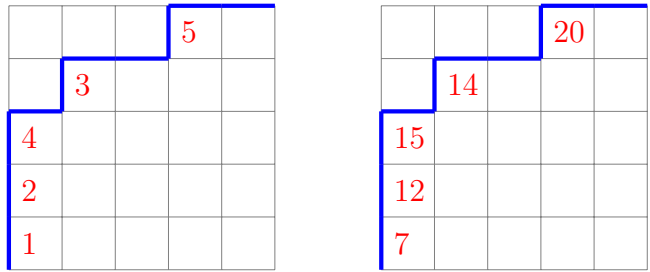
\begin{figure}
  \centering
\begin{tikzpicture}[scale=.7]
  \begin{scope}
\draw[step=1cm,gray,very thin] (0,0) grid (5,5); 
\draw[ultra thick,blue] (0,0)--(0,1) node[red,midway,right] {$1$};
\draw[ultra thick,blue] (0,1)--(0,2) node[red,midway,right] {$2$};
\draw[ultra thick,blue] (0,2)--(0,3) node[red,midway,right] {$4$};
\draw[ultra thick,blue] (0,3)--(1,3);
\draw[ultra thick,blue] (1,3)--(1,4) node[red,midway,right] {$3$};
\draw[ultra thick,blue] (1,4)--(2,4);
\draw[ultra thick,blue] (2,4)--(3,4);
\draw[ultra thick,blue] (3,4)--(3,5) node[red,midway,right] {$5$};
\draw[ultra thick,blue] (3,5)--(4,5);
\draw[ultra thick,blue] (4,5)--(5,5);
\end{scope}
  \begin{scope}[shift={(7,0)}]
\draw[step=1cm,gray,very thin] (0,0) grid (5,5); 
\draw[ultra thick,blue] (0,0)--(0,1) node[red,midway,right] {$7$};
\draw[ultra thick,blue] (0,1)--(0,2) node[red,midway,right] {$12$};
\draw[ultra thick,blue] (0,2)--(0,3) node[red,midway,right] {$15$};
\draw[ultra thick,blue] (0,3)--(1,3);
\draw[ultra thick,blue] (1,3)--(1,4) node[red,midway,right] {$14$};
\draw[ultra thick,blue] (1,4)--(2,4);
\draw[ultra thick,blue] (2,4)--(3,4);
\draw[ultra thick,blue] (3,4)--(3,5) node[red,midway,right] {$20$};
\draw[ultra thick,blue] (3,5)--(4,5);
\draw[ultra thick,blue] (4,5)--(5,5);
\end{scope}

  \end{tikzpicture}

\caption{On the left, the parking function $f\in\PF(5)$ whose function form is $(1,1,2,1,4)$. On the right is a parking function in $\PF(A)$, $A=\{7,12,14,15,20\}$.}
\label{fig:pfAndDiagram}.
\end{figure}


\subsection{Diagonal inversion ($\dinv$) statistics}

Let $f \in \PF(n)$ be a parking function. The \mydef{row} $\rowa f i$
of $i\in[n]$ is the row of $i$ in $f$, counting from the bottom, and
its \mydef{column} $\cola f i$ counting from the left. For example, in
Figure~\ref{fig:pfAndDiagram} on the left, we have $\rowa f 2 = 2$ and $\cola f 2
= 1$. The \mydef{diagonal} of an element $i\in[n]$ is $\diaga f
i=\rowa f i-\cola f i$.

We will also consider the set \mydef{diagonal} $d$ for $d\in[n]$. It is the set
$$\{i\in[n]:\diaga f i=d\}.$$

Again referring to Figure~\ref{fig:pfAndDiagram} on the left,
diagonal $0$ of $f$ is $\{1\}$, diagonal $1$ of $f$ is $\{2,5\}$, diagonal $2$ of $f$ is $\{3,4\}$,
and
all higher diagonals are empty. We denote the $\max_{x\in A}\diaga f x$ by $\maxdiag f$. All parking functions have nonempty diagonal $0$.

There are two types of \mydef{dinv pairs}. If $i,j\in[n]$ with $i<j$, $\diaga
f i=\diaga f j$, and $\cola f i<\cola f j$ ($i$ to left, same diagonal), then $(i,j)$ is a primary
dinv pair and $\pridinv(f)$ is defined as the number of such
pairs. If $i,j\in[n]$ with $i<j$, $\diaga f i=\diaga f j-1$, and
$\cola f i>\cola f j$ ($i$ strictly to right, lower adjacent diagonal), then $(j,i)$ is a secondary dinv pair and the
number of these pairs is $\secdinv(f)$. Please see Figure~\ref{fig:dinvschematic}. Primary and secondary dinv
were introduced in \cite{gxz2014} and their sum is dinv. We study
parking functions with \mydef{zero secondary dinv}: $\secdinv(f)=0$ and we denote this subset of $\PF(n)$ by $\PFZ(n)$. Parking functions on a domain $A$ with zero secondary dinv are denoted $\PFZ(A)$.

In Figure ~\ref{fig:pfAndDiagram} on the left, there is one primary dinv pair (2, 5), and no secondary dinv pairs; therefore
$\pridinv(f) = 1$ and $\secdinv(f) = 0$, and we have $f \in \PFZ(5)$.

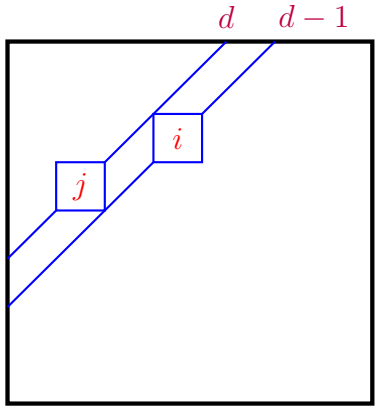
\begin{figure}
  \centering
  \begin{tikzpicture}[scale=.8]
    \def\h{6}
    \def\s{.8}
    \def\d{3}
    \def\k{3}
    \def\kj{1}
    \draw[ultra thick] (0,0)--(0,\h)--(\h,\h)--(\h,0)--cycle;
    \draw[thick,blue] (0,\d*\s-\s)--(\k*\s, \k*\s+\d*\s-1*\s);
    \node[red] (i) at (\k*\s+.5*\s, \k*\s+\d*\s-.5*\s) {$i$};
    \draw[thick,blue] (\k*\s, \k*\s+\d*\s-1*\s)--(\k*\s, \k*\s+\d*\s)--(\k*\s+\s, \k*\s+\d*\s)--(\k*\s+\s, \k*\s+\d*\s-1*\s)--cycle;
    \draw[thick,blue] (\k*\s+\s, \k*\s+\d*\s)--(\h-\d*\s+\s,\h);
    \node[purple] at (\h-\d*\s+\s+.8*\s,\h+.5*\s)  {$d-1$};
    
    \draw[thick,blue] (0,\d*\s)--(\kj*\s, \kj*\s+\d*\s);
    \node[red] (i) at (\kj*\s+.5*\s, \kj*\s+\d*\s+.5*\s) {$j$};
    \draw[thick,blue] (\kj*\s, \kj*\s+\d*\s)--(\kj*\s, \kj*\s+\d*\s+\s)--(\kj*\s+\s, \kj*\s+\d*\s+\s)--(\kj*\s+\s, \kj*\s+\d*\s)--cycle;
    \draw[thick,blue] (\kj*\s+\s, \kj*\s+\d*\s+\s)--(\h-\d*\s,\h);
    \node [purple] at(\h-\d*\s,\h+.5*\s) {$d$};
    
    \end{tikzpicture}
  \caption{If $i<j$, $\diaga f i=\diaga f j-1$, and $\cola f i>\cola f j$, then $(j,i)$ is a secondary dinv pair for $f$.}
  \label{fig:dinvschematic}
\end{figure}

\subsection{Ordered cycle decompositions}\label{subsec:ocd} We define the second main object of this note and prove Proposition~\ref{prop:ocd_rec}.
\begin{definition}
    \label{def:ocd}
    An \mydef{ordered cycle decomposition} of a permutation of $n$ is a sequence $(\sigma_1,\sigma_2,\ldots,\sigma_k)$ of nonempty, disjoint cycles whose product $\sigma_1\cdot\sigma_2\cdots\sigma_k$ is a permutation of $n$. 
\end{definition}

As in the introduction, we denote the number of ordered cycle decompositions of permutations of $n$ by $\numocd(n).$ For example, $\numocd(2)=3$ as $\{ (12),(1)(2) , (2)(1) \}$ is the set of ordered cycle decompositions of permutations of $2$, and $\numocd(3) = 14$ as 
$$
\begin{array}{l}
\{ (123), (132), (12)(3), (3)(12), (13)(2), (2)(13), (23)(1), (1)(23), \\
(1)(2)(3), (1)(3)(2), (2)(1)(3), (2)(3)(1), (3)(1)(2), (3)(2)(1) \}
\end{array}
$$
is the set of ordered cycle decompositions of permutations of $3$.

The expression
\begin{equation}
\numocd(n)=\sum_{k=0}^nc(n,k)k!,\end{equation} where $c(n,k)$ is the signless Stirling number of the first kind, also appears in \cite{AA}.

 We claim in Proposition~\ref{prop:ocd_rec} that

\begin{equation}\label{eq:ocd_rec}
\numocd(n)= \sum_{k=1}^{n} \dfrac{(n-1)!}{ (n-k)! } \sum_{\ell=0}^{n-k} \binom{n-k}{\ell}  \numocd(\ell)  \numocd({n-k-\ell}) .
\end{equation}

\begin{proof}[Proof of Proposition~\ref{prop:ocd_rec}]
Suppose there are $k$ elements in the cycle which contains $1$ and
$\ell$ elements which precede this
cycle, where $0\le\ell\le n-k$. There are $\binom{n-1}{k-1}\cdot (k-1)!=\frac{(n-1)!}{(k-1)!}$
such cycles, $\binom{n-k}{\ell}\cdot\numocd(\ell)$ ways to pick and
arrange the elements which precede the cycle, and $\numocd(n-k-\ell)$ ways
to arrange the elements which follow the cycle. Finally, sum over $k$
and $\ell$.

\end{proof}

\section{The map $\mainmap$}
\label{sec:map}
Fix a positive integer $n$.  Let $A,B,E\subset[n]$ be such that $1\in E$ and $[n]$ is the disjoint union of $A$, $B$, and $E$.
Set
\begin{equation}\label{def:decomp}
  \decomp(A,B,E)=\{(f,g,D)\},
\end{equation}
  where $f\in\PFZ(A)$, $g\in\PFZ(B)$,
and $D=(e_1,e_2,\ldots,e_{|E|})$ is an ordering of $E$ such
$e_1=1$. We
set

\begin{equation}
  \label{eq:decomp}
  \decomp(n)=\bigcup_{\substack{k,\ell\in\Z_{\ge0}\\k+\ell\le
      n}}\bigcup_{\substack{|A|=\ell\\|B|=k}}\decomp(A,B,E),
 \end{equation} 
where $A$ and  $B$ in \eqref{eq:decomp} are disjoint subsets of $[n]$ which do not contain $1$, $E=[n]\setminus A\cup B$, and $\decomp(A,B,E)$ is as in \eqref{def:decomp}. 


\subsection{Insertion}
\label{subsec:insertion}
Let $A\subset\Z$, let $f \in \PFZ (A)$, let $x \not\in A$, and let $k$ be a positive integer. We define the insertion of number $x$ into $f$ in place $k$ as follows and denote the result by $\hat f = \Gamma (f,x,k)$. Here we use the function form of $f$. For $a\in A\cup\{x\}$ we define

\begin{equation}
\label{Insertion}
{\hat f}(a) =
\begin{cases}
 f(a) & \text{ if $f(a) < k$  or  $f(a)=k$ and  $a<x$,}  \\ 
 f(a)+1 & \text{ if $f(a) > k$  or  $f(a)=k$ and  $a>x$,} \\
k &\text{if $a=x$.}
\end{cases}
\end{equation}

Insertion can be seen in terms of the diagram. It splits the diagram of $f$ at column $k$, with the portion to the right of $a$'s cell or directly above $a$'s cell shifting to the right. A row and column will be added to $f$'s diagram form.

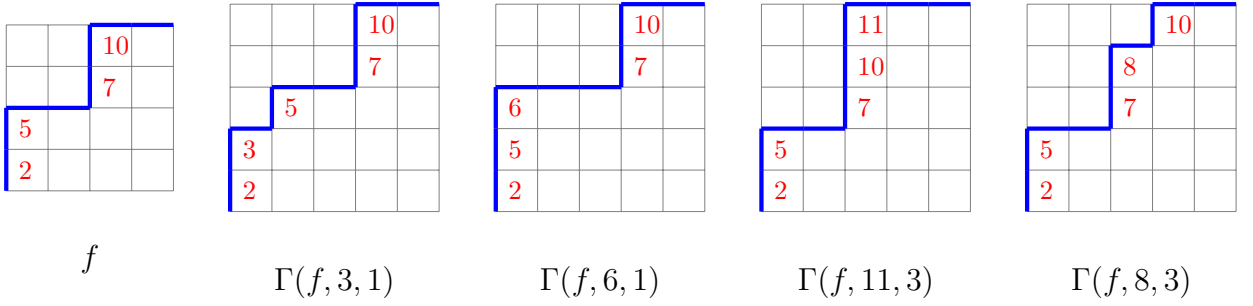
\begin{figure}
  \centering
  \begin{tikzpicture}[node distance=.4cm]
    \def\s{.55}
    \node (A) at (0,0) {\begin{tikzpicture}[scale=\s,font=\footnotesize]
   \begin{scope}
\draw[step=1cm,gray,very thin] (0,0) grid (4,4); 
\draw[ultra thick,blue] (0,0)--(0,1) node[red,midway,right] {$2$};
\draw[ultra thick,blue] (0,1)--(0,2) node[red,midway,right] {$5$};
\draw[ultra thick,blue] (0,2)--(1,2);
\draw[ultra thick,blue] (1,2)--(2,2);
\draw[ultra thick,blue] (2,2)--(2,3) node[red,midway,right] {$7$};
\draw[ultra thick,blue] (2,3)--(2,4) node[red,midway,right] {$10$};
\draw[ultra thick,blue] (2,4)--(3,4);
\draw[ultra thick,blue] (3,4)--(4,4);
\end{scope}
     \end{tikzpicture}};
    \node (A1) [below=of A]{$f$};
    \node  (B) [right=of A]{\begin{tikzpicture}[scale=\s,font=\footnotesize]
\begin{scope}
\draw[step=1cm,gray,very thin] (0,0) grid (5,5); 
\draw[ultra thick,blue] (0,0)--(0,1) node[red,midway,right] {$2$};
\draw[ultra thick,blue] (0,1)--(0,2) node[red,midway,right] {$3$};
\draw[ultra thick,blue] (0,2)--(1,2);
\draw[ultra thick,blue] (1,2)--(1,3) node[red,midway,right] {$5$};
\draw[ultra thick,blue] (1,3)--(2,3);
\draw[ultra thick,blue] (2,3)--(3,3);
\draw[ultra thick,blue] (3,3)--(3,4) node[red,midway,right] {$7$};
\draw[ultra thick,blue] (3,4)--(3,5) node[red,midway,right] {$10$};
\draw[ultra thick,blue] (3,5)--(4,5);
\draw[ultra thick,blue] (4,5)--(5,5);
\end{scope}
         \end{tikzpicture}};
    \node (B1)[below=of B] {$\Gamma(f,3,1)$};
    \node  (C) [right=of B]{\begin{tikzpicture}[scale=\s,font=\footnotesize]
   \begin{scope}
\draw[step=1cm,gray,very thin] (0,0) grid (5,5); 
\draw[ultra thick,blue] (0,0)--(0,1) node[red,midway,right] {$2$};
\draw[ultra thick,blue] (0,1)--(0,2) node[red,midway,right] {$5$};
\draw[ultra thick,blue] (0,2)--(0,3) node[red,midway,right] {$6$};
\draw[ultra thick,blue] (0,3)--(1,3);
\draw[ultra thick,blue] (1,3)--(2,3);
\draw[ultra thick,blue] (2,3)--(3,3);
\draw[ultra thick,blue] (3,3)--(3,4) node[red,midway,right] {$7$};
\draw[ultra thick,blue] (3,4)--(3,5) node[red,midway,right] {$10$};
\draw[ultra thick,blue] (3,5)--(4,5);
\draw[ultra thick,blue] (4,5)--(5,5);
\end{scope}
      \end{tikzpicture}};
    \node (C1)[below=of C] {$\Gamma(f,6,1)$};
    \node  (D) [right=of C]{\begin{tikzpicture}[scale=\s,font=\footnotesize]
        \begin{scope}
\draw[step=1cm,gray,very thin] (0,0) grid (5,5); 
\draw[ultra thick,blue] (0,0)--(0,1) node[red,midway,right] {$2$};
\draw[ultra thick,blue] (0,1)--(0,2) node[red,midway,right] {$5$};
\draw[ultra thick,blue] (0,2)--(1,2);
\draw[ultra thick,blue] (1,2)--(2,2);
\draw[ultra thick,blue] (2,2)--(2,3) node[red,midway,right] {$7$};
\draw[ultra thick,blue] (2,3)--(2,4) node[red,midway,right] {$10$};
\draw[ultra thick,blue] (2,4)--(2,5) node[red,midway,right] {$11$};
\draw[ultra thick,blue] (2,5)--(3,5);
\draw[ultra thick,blue] (3,5)--(4,5);
\draw[ultra thick,blue] (4,5)--(5,5);
\end{scope}
    \end{tikzpicture}};
    \node (D1)[below=of D] {$\Gamma(f,11,3)$};
    \node  (E) [right=of D]{\begin{tikzpicture}[scale=\s,font=\footnotesize]
\begin{scope}
\draw[step=1cm,gray,very thin] (0,0) grid (5,5); 
\draw[ultra thick,blue] (0,0)--(0,1) node[red,midway,right] {$2$};
\draw[ultra thick,blue] (0,1)--(0,2) node[red,midway,right] {$5$};
\draw[ultra thick,blue] (0,2)--(1,2);
\draw[ultra thick,blue] (1,2)--(2,2);
\draw[ultra thick,blue] (2,2)--(2,3) node[red,midway,right] {$7$};
\draw[ultra thick,blue] (2,3)--(2,4) node[red,midway,right] {$8$};
\draw[ultra thick,blue] (2,4)--(3,4);
\draw[ultra thick,blue] (3,4)--(3,5) node[red,midway,right] {$10$};
\draw[ultra thick,blue] (3,5)--(4,5);
\draw[ultra thick,blue] (4,5)--(5,5);
\end{scope}
\end{tikzpicture}};
    \node (E1)[below=of E] {$\Gamma(f,8,3)$};
  
\end{tikzpicture}
\caption{The parking function $f$ and the results of various insertions.}
\label{fig:insertion}
\end{figure}

For the map $\mainmap$, we'll need to insert without increasing secondary dinv.
\begin{definition}
  Let $h$ be a parking function on $A$, let $x\notin A$, let $c,d\le|A|$. Denote the elements in diagonal $d-1$ which are strictly to the right of column $c$ by $y_1,y_2,\ldots,y_{m}$. They are ordered so that
$$c<\cola h {y_1}<\cola h {y_2}<\cdots<\cola h {y_m}.$$ Let $\ell^*$ be the minimum $\ell$ such that for all $\ell>\ell^*$, $y_{\ell}>x$. If $x<y_1$, then $\ell^*=0$. 
  The \mydef{$(c,d)$-insertion} of $x$ in $h$ inserts $x$ on diagonal $d$ and in a column at least $c$ of $h$ and results in the parking function $h^*=\Gamma(h,x,c^*)$, where
  \begin{equation*}
    c^*=\begin{cases}\cola h {y_{\ell^*}}&\text{if $\ell^*>0$}\\
      c+1&\text{if $\ell^*=0$.}
    \end{cases}
\end{equation*}
    
  \end{definition}
In general, we refer to $(c,d)$-insertion as \mydef{special insertion.} 

\begin{claim}
  \label{claim:specinsert} Let $f\in\PFZ(A)$ and $x\notin A$. Suppose $c$ and $d$ satisfy the following conditions.
  \begin{enumerate}
  \item \label{cond:dcond} We have $d=\maxdiag f$ and for all $y$ such that $\cola f y>c$, we have $d>\diaga f y$; and 
  \item \label{cond:ccond} there is a label which is in column $c$ and diagonal $d$ and it's at the top of its column.  
\end{enumerate}

  If $h$ is the result of a $(c,d)$-insertion of $x$ into $f$, then $\secdinv(h)=\secdinv(f)$. What's more, $\diaga h x=d$. 
  \end{claim}

\begin{figure}
  \centering
  \begin{tikzpicture}[scale=.8]
    \def\h{6}
    \def\s{.8}
    \def\d{3}
    \def\k{3}
    \def\kj{1}
    \def\kja{3}
    \draw[ultra thick] (0,0)--(0,\h)--(\h,\h)--(\h,0)--cycle;
    
    \draw[thick,blue] (0,\d*\s)--(\kj*\s, \kj*\s+\d*\s);
    \node[red] (i) at (\kj*\s+.5*\s, \kj*\s+\d*\s+.5*\s) {$u$};
    \draw[thick,blue] (\kj*\s, \kj*\s+\d*\s)--(\kj*\s, \kj*\s+\d*\s+\s)--(\kj*\s+\s, \kj*\s+\d*\s+\s)--(\kj*\s+\s, \kj*\s+\d*\s)--cycle;
    \draw[thick,blue] (\kj*\s+\s, \kj*\s+\d*\s+\s)--(\kja*\s, \kja*\s+\d*\s);
    
 \node[red] (i) at (\kja*\s+.5*\s, \kja*\s+\d*\s+.5*\s) {$v$};
    \draw[thick,blue] (\kja*\s, \kja*\s+\d*\s)--(\kja*\s, \kja*\s+\d*\s+\s)--(\kja*\s+\s, \kja*\s+\d*\s+\s)--(\kja*\s+\s, \kja*\s+\d*\s)--cycle;
    \draw[thick,blue] (\kja*\s+\s, \kja*\s+\d*\s+\s)--(\h-\d*\s,\h);
    \draw[thick,olive] (\kja*\s+.5*\s, \kja*\s+\d*\s)--(\kja*\s+.5*\s,0);
    \node [purple] at (\kja*\s+.5*\s,-.5*\s) {$c$};
    \node [purple] at(\h-\d*\s,\h+.5*\s) {$d=\maxdiag f$};
    
    \end{tikzpicture}
  \caption{Condition~\eqref{cond:dcond} of Claim~\ref{claim:specinsert} says that $d$ is the highest index of a diagonal which has a label in it and this diagonal has no labels on it to the right of column $c$. Condition~\eqref{cond:ccond} says that column $c$ has a label on diagonal $d$.}
  \label{fig:specinsertschematic}
\end{figure}
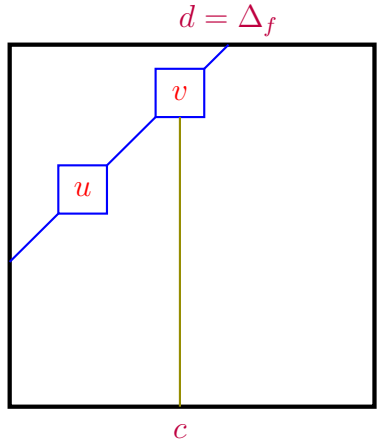
  
\begin{proof}

  First, we show that $\diaga h x = d$. Suppose $\ell^*>0$. Then $y_{\ell^*}<x$, $y_{\ell^*}$ is in diagonal $d-1$, and $c^*=\cola h {y_{\ell^*}}$, so $x$ will be placed directly on top of $y_{\ell^*}$ in diagonal $d$. Now consider the case $\ell^*=0$, so that $c^*=c+1$. By condition \eqref{cond:ccond}, $x$ must be placed in the row above the row that the largest element of column $c$ is in. Therefore, again, $x$ is in diagonal $d$ of $h$.

Recall that $(j,i)$, where $i$ and $j$ are labels in $h$,  is a secondary dinv pair if $i<j$, $\diaga h i=\diaga h j-1$, and $\cola h i>\cola h j$. By condition \eqref{cond:dcond}, there can be no secondary dinv pair $(j,x)$ and by the definition of $\ell^*$, there can be no secondary dinv pair $(x,i)$.   
  
  \end{proof}

\subsection{Definition of $\mainmap:\decomp(n)\to\PFZ(n)$}
\label{subsec:map}

Let $(f,g,D)\in\decomp(n)$, where $f\in\PFZ(A)$, $g\in\PFZ(B)$,
$D=(e_1=1,e_2,\ldots,e_{s})$, and $[n]$ is the disjoint union of $A$,
$B$, and $\{e_1,e_2,\ldots,e_s\}$. The map $\mainmap:(f,g,D)\mapsto h$
consists of three phases and produces a sequence of functions $f=
h_{0}, h_{1}, \ldots , h_{N} = h$ with all $h_{k}$ having zero secondary
dinv.

In the first phase, we will use special insertion to insert $D$ into
$f$. All elements of $D$ are inserted into diagonal $\maxdiag f$. Let
$\mcc$ be the column of the entry in diagonal $\maxdiag f$ with the
largest column index. We start with $h_0=f$ and $(\mcc,\maxdiag
f)$-insert $1$ into $h_0$, obtaining $h_1$. Notice that 1 will be the only label in its row and column. At step $i$, $1<i\le s$, we
let $\mcc=\cola {h_{i-1}} {e_{i-1}}$ and $(\mcc,\maxdiag f)$-insert
$e_i$ into $h_{i-1}$. Conditions \eqref{cond:dcond} and
\eqref{cond:ccond} of Claim~\ref{claim:specinsert} are satisfied, so
we have not added any secondary dinv. At the end of the first phase,
we have $h_0=f,h_1,\ldots,h_s$.

In the second phase, we special insert the entries of the main
diagonal (diagonal $0$) of $g$ into $h_s$. Let $z_1,z_2,\ldots,z_t$ be
the entries of the main diagonal of $g$, ordered so that $\cola g
{z_1}<\cola g {z_2}<\cdots<\cola g {z_t}$. This phase is similar to
the first phase, except that we now insert into diagonal
$\mcd=(\maxdiag f+1)=(\maxdiag {h_s}+1)$ of $h_s$. 
We begin with
$\mcc=\cola {h_s} 1-1$ and we $(\mcc,\mcd)$-insert $z_1$ into $h_s$ to obtain $h_s+1$. Notice that we cannot apply Claim~\ref{claim:specinsert} for the first insertion of phase two, because its condition \eqref{cond:ccond} is not satisfied. However, the $\ell^*$ used in special insertion will be positive in this case, since 1 is on diagonal $\mcd-1=\maxdiag {h_s}$. Since $\ell^*>0$, the first entry will be placed on top of $y_{\ell^*}$ in diagonal $\mcd=\diag {h_s}+1$. As in the proof of Claim~\ref{claim:specinsert}, since $\diaga {h_{s+1}} {z_1}=\maxdiag{h_s}$, there can be no secondary dinv pair $(j,z_1)$ and by the definition of $\ell^*$, there can be no secondary dinv pair $(z_1,i)$. At
step $i$, $s+1<i<s+t$, let $\mcc=\cola{h_{i-1}}{z_{i-1-s}}$ and
$(\mcc,\mcd)$-insert $z_{i-s}$ into $h_{i-1}$. At the end of this
phase, we have $h_0,\ldots,h_s,\ldots,h_{s+t}$. Again, conditions
\eqref{cond:dcond} and \eqref{cond:ccond} of
Claim~\ref{claim:specinsert} are satisfied at each step.

We are now in the third phase. Let $z_1,\ldots,z_t$ be as defined in the second phase. The \mydef{block} $B_i$ of $z_i$
is the set of labels $x$ in $g$ such  that $\diaga g x >0$ and $\cola g {z_i}\le\cola g x<\cola
g {z_{i+1}}$, where we set $\cola g {z_{t+1}}$ to
be $|\dom(g)|+1$ for ease of notation. Please see Figure~\ref{fig:mainmap}. In the third phase, we insert the blocks of $g$ into $h_{s+t}$. There
are $t$ blocks of $g$: one for each element $z_i$ from diagonal
$0$ of $g$.  

\begin{figure}
  \centering
  \begin{tikzpicture}[node distance=.4cm]
    \def\s{.55}

    \node  at (0,0) {\begin{tikzpicture}[scale=\s,font=\footnotesize]
\begin{scope}
\draw[step=1cm,gray,very thin] (0,0) grid (9,9); 
\draw[ultra thick,blue] (0,0)--(0,1) node[red,midway,right] {$2$};
\draw[ultra thick,blue] (0,1)--(0,2) node[red,midway,right] {$4$};
\draw[ultra thick,blue] (0,2)--(0,3) node[red,midway,right] {$12$};
\draw[ultra thick,blue] (0,3)--(1,3);
\draw[ultra thick,blue] (1,3)--(2,3);
\draw[ultra thick,blue] (2,3)--(2,4) node[red,midway,right] {$13$};
\draw[ultra thick,blue] (2,4)--(3,4);
\draw[ultra thick,blue] (3,4)--(4,4);
\draw[ultra thick,blue] (4,4)--(4,5) node[red,midway,right] {$16$};
\draw[ultra thick,blue] (4,5)--(4,6) node[red,midway,right] {$17$};
\draw[ultra thick,blue] (4,6)--(5,6);
\draw[ultra thick,blue] (5,6)--(5,7) node[red,midway,right] {$19$};
\draw[ultra thick,blue] (5,7)--(6,7);
\draw[ultra thick,blue] (6,7)--(7,7);
\draw[ultra thick,blue] (7,7)--(7,8) node[red,midway,right] {$21$};
\draw[ultra thick,blue] (7,8)--(8,8);
\draw[ultra thick,blue] (8,8)--(8,9) node[red,midway,right] {$10$};
\draw[ultra thick,blue] (8,9)--(9,9);

\draw[ultra thick, red, decorate, decoration={bumps}] (0,1)--(0,4)--(3,4)--(3,1)--cycle;
\draw[ultra thick, red, decorate, decoration={bumps}] (4,5)--(4,7)--(6,7)--(6,5)--cycle;

\end{scope}
\end{tikzpicture}};
  \end{tikzpicture}
\caption{Parking function $g$ from Figure~\ref{fig:mainmapA} with blocks drawn in.}
\label{fig:mainmap}
\end{figure}
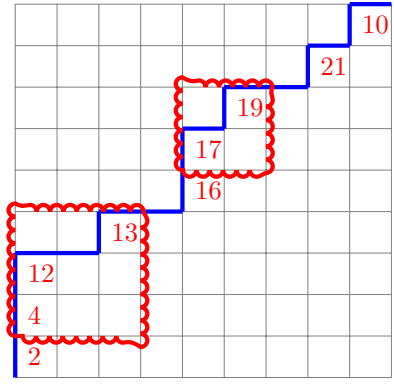
The \mydef{width} of $B_i$, $\WD(B_{i})$, is $\cola g {z_{i+1}}-\cola g {z_i}+1$. The number of rows occupied in $g$ by elements of $B_i$ is also $\WD(B_i)$. For each $B_i$, we expand the diagram of $h_{s+t+i-1}$ by $\WD(B_{i})$ columns inserted after $\cola {h_{s+t+i-1}} {z_i}$, then add the labels from $B_i$ so that their relative position to $z_i$ (number of rows above, number of columns to the right) is as it was in $g$. Let $x$ be a label added in the third phase from block $B_i$. Note that the preservation of relative positions forces
$$\diaga h x=\diaga g x +\diaga h {z_i}=\diaga g x+\mcd>\mcd.$$

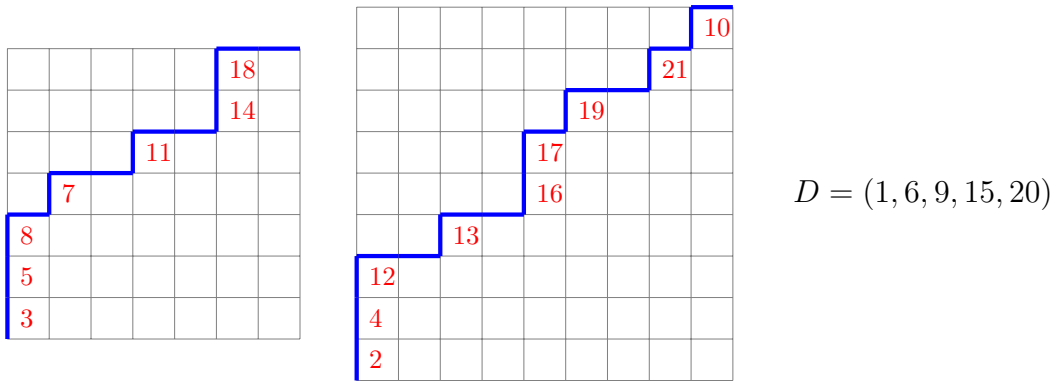
\begin{figure}
  \centering
  \begin{tikzpicture}[node distance=.4cm]
    \def\s{.55}
    \node (f) at (0,0) {\begin{tikzpicture}[scale=\s,font=\footnotesize]
\begin{scope}
\draw[step=1cm,gray,very thin] (0,0) grid (7,7); 
\draw[ultra thick,blue] (0,0)--(0,1) node[red,midway,right] {$3$};
\draw[ultra thick,blue] (0,1)--(0,2) node[red,midway,right] {$5$};
\draw[ultra thick,blue] (0,2)--(0,3) node[red,midway,right] {$8$};
\draw[ultra thick,blue] (0,3)--(1,3);
\draw[ultra thick,blue] (1,3)--(1,4) node[red,midway,right] {$7$};
\draw[ultra thick,blue] (1,4)--(2,4);
\draw[ultra thick,blue] (2,4)--(3,4);
\draw[ultra thick,blue] (3,4)--(3,5) node[red,midway,right] {$11$};
\draw[ultra thick,blue] (3,5)--(4,5);
\draw[ultra thick,blue] (4,5)--(5,5);
\draw[ultra thick,blue] (5,5)--(5,6) node[red,midway,right] {$14$};
\draw[ultra thick,blue] (5,6)--(5,7) node[red,midway,right] {$18$};
\draw[ultra thick,blue] (5,7)--(6,7);
\draw[ultra thick,blue] (6,7)--(7,7);
\end{scope}
\end{tikzpicture}};
    \node  (g) [right=of f]{\begin{tikzpicture}[scale=\s,font=\footnotesize]
\begin{scope}
\draw[step=1cm,gray,very thin] (0,0) grid (9,9); 
\draw[ultra thick,blue] (0,0)--(0,1) node[red,midway,right] {$2$};
\draw[ultra thick,blue] (0,1)--(0,2) node[red,midway,right] {$4$};
\draw[ultra thick,blue] (0,2)--(0,3) node[red,midway,right] {$12$};
\draw[ultra thick,blue] (0,3)--(1,3);
\draw[ultra thick,blue] (1,3)--(2,3);
\draw[ultra thick,blue] (2,3)--(2,4) node[red,midway,right] {$13$};
\draw[ultra thick,blue] (2,4)--(3,4);
\draw[ultra thick,blue] (3,4)--(4,4);
\draw[ultra thick,blue] (4,4)--(4,5) node[red,midway,right] {$16$};
\draw[ultra thick,blue] (4,5)--(4,6) node[red,midway,right] {$17$};
\draw[ultra thick,blue] (4,6)--(5,6);
\draw[ultra thick,blue] (5,6)--(5,7) node[red,midway,right] {$19$};
\draw[ultra thick,blue] (5,7)--(6,7);
\draw[ultra thick,blue] (6,7)--(7,7);
\draw[ultra thick,blue] (7,7)--(7,8) node[red,midway,right] {$21$};
\draw[ultra thick,blue] (7,8)--(8,8);
\draw[ultra thick,blue] (8,8)--(8,9) node[red,midway,right] {$10$};
\draw[ultra thick,blue] (8,9)--(9,9);
\end{scope}
\end{tikzpicture}};
\node  (D) [right=of g]{$D=(1,6,9,15,20)$};
  
\end{tikzpicture}
\caption{Parking functions $f$ and $g$, as well as a sequence $D$. See Example~\ref{ex:mainmap}: $(f,g,D)\in\decomp(21)$.}
\label{fig:mainmapA}
\end{figure}

\begin{figure}
  \centering
  \begin{tikzpicture}[node distance=.4cm]
    \def\s{.55}

    \node  (fD) at (0,0){\begin{tikzpicture}[scale=\s,font=\footnotesize]
\begin{scope}
\draw[step=1cm,gray,very thin] (0,0) grid (12,12); 
\draw[ultra thick,blue] (0,0)--(0,1) node[red,midway,right] {$3$};
\draw[ultra thick,blue] (0,1)--(0,2) node[red,midway,right] {$5$};
\draw[ultra thick,blue] (0,2)--(0,3) node[red,midway,right] {$8$};
\draw[ultra thick,blue] (0,3)--(1,3);
\draw[ultra thick,blue] (1,3)--(1,4) node[red,midway,right] {$7$};
\draw[ultra thick,blue] (1,4)--(2,4);
\draw[ultra thick,blue] (2,4)--(2,5) node[red,midway,right] {\circled{$1$}};
\draw[ultra thick,blue] (2,5)--(3,5);
\draw[ultra thick,blue] (3,5)--(3,6) node[red,midway,right] {\circled{$6$}};
\draw[ultra thick,blue] (3,6)--(4,6);
\draw[ultra thick,blue] (4,6)--(4,7) node[red,midway,right] {\circled{$9$}};
\draw[ultra thick,blue] (4,7)--(5,7);
\draw[ultra thick,blue] (5,7)--(6,7);
\draw[ultra thick,blue] (6,7)--(6,8) node[red,midway,right] {$11$};
\draw[ultra thick,blue] (6,8)--(6,9) node[red,midway,right] {\circled{$15$}};
\draw[ultra thick,blue] (6,9)--(7,9);
\draw[ultra thick,blue] (7,9)--(8,9);
\draw[ultra thick,blue] (8,9)--(9,9);
\draw[ultra thick,blue] (9,9)--(9,10) node[red,midway,right] {$14$};
\draw[ultra thick,blue] (9,10)--(9,11) node[red,midway,right] {$18$};
\draw[ultra thick,blue] (9,11)--(9,12) node[red,midway,right] {\circled{$20$}};
\draw[ultra thick,blue] (9,12)--(10,12);
\draw[ultra thick,blue] (10,12)--(11,12);
\draw[ultra thick,blue] (11,12)--(12,12);
\end{scope}
\end{tikzpicture}};

    \node  (phaseII)[right=of fD]{\begin{tikzpicture}[scale=\s,font=\footnotesize]
\begin{scope}
\draw[step=1cm,gray,very thin] (0,0) grid (16,16); 
\draw[ultra thick,blue] (0,0)--(0,1) node[red,midway,right] {$3$};
\draw[ultra thick,blue] (0,1)--(0,2) node[red,midway,right] {$5$};
\draw[ultra thick,blue] (0,2)--(0,3) node[red,midway,right] {$8$};
\draw[ultra thick,blue] (0,3)--(1,3);
\draw[ultra thick,blue] (1,3)--(1,4) node[red,midway,right] {$7$};
\draw[ultra thick,blue] (1,4)--(2,4);
\draw[ultra thick,blue] (2,4)--(2,5) node[red,midway,right] {$1$};
\draw[ultra thick,blue] (2,5)--(2,6) node[red,midway,right] {\circled{$2$}};
\draw[ultra thick,blue] (2,6)--(3,6);
\draw[ultra thick,blue] (3,6)--(4,6);
\draw[ultra thick,blue] (4,6)--(4,7) node[red,midway,right] {$6$};
\draw[ultra thick,blue] (4,7)--(5,7);
\draw[ultra thick,blue] (5,7)--(5,8) node[red,midway,right] {$9$};
\draw[ultra thick,blue] (5,8)--(6,8);
\draw[ultra thick,blue] (6,8)--(7,8);
\draw[ultra thick,blue] (7,8)--(7,9) node[red,midway,right] {$11$};
\draw[ultra thick,blue] (7,9)--(7,10) node[red,midway,right] {$15$};
\draw[ultra thick,blue] (7,10)--(7,11) node[red,midway,right] {\circled{$16$}};
\draw[ultra thick,blue] (7,11)--(8,11);
\draw[ultra thick,blue] (8,11)--(9,11);
\draw[ultra thick,blue] (9,11)--(10,11);
\draw[ultra thick,blue] (10,11)--(11,11);
\draw[ultra thick,blue] (11,11)--(11,12) node[red,midway,right] {$14$};
\draw[ultra thick,blue] (11,12)--(11,13) node[red,midway,right] {$18$};
\draw[ultra thick,blue] (11,13)--(11,14) node[red,midway,right] {$20$};
\draw[ultra thick,blue] (11,14)--(11,15) node[red,midway,right] {\circled{$21$}};
\draw[ultra thick,blue] (11,15)--(12,15);
\draw[ultra thick,blue] (12,15)--(12,16) node[red,midway,right] {\circled{$10$}};
\draw[ultra thick,blue] (12,16)--(13,16);
\draw[ultra thick,blue] (13,16)--(14,16);
\draw[ultra thick,blue] (14,16)--(15,16);
\draw[ultra thick,blue] (15,16)--(16,16);
\end{scope}
\end{tikzpicture}};

\end{tikzpicture}
\caption{The results after phases 1 and 2 of the map $\mainmap$ applied to $(f,g,D)$ from Figure~\ref{fig:mainmapA}.}
\label{fig:mainmapB}
\end{figure}
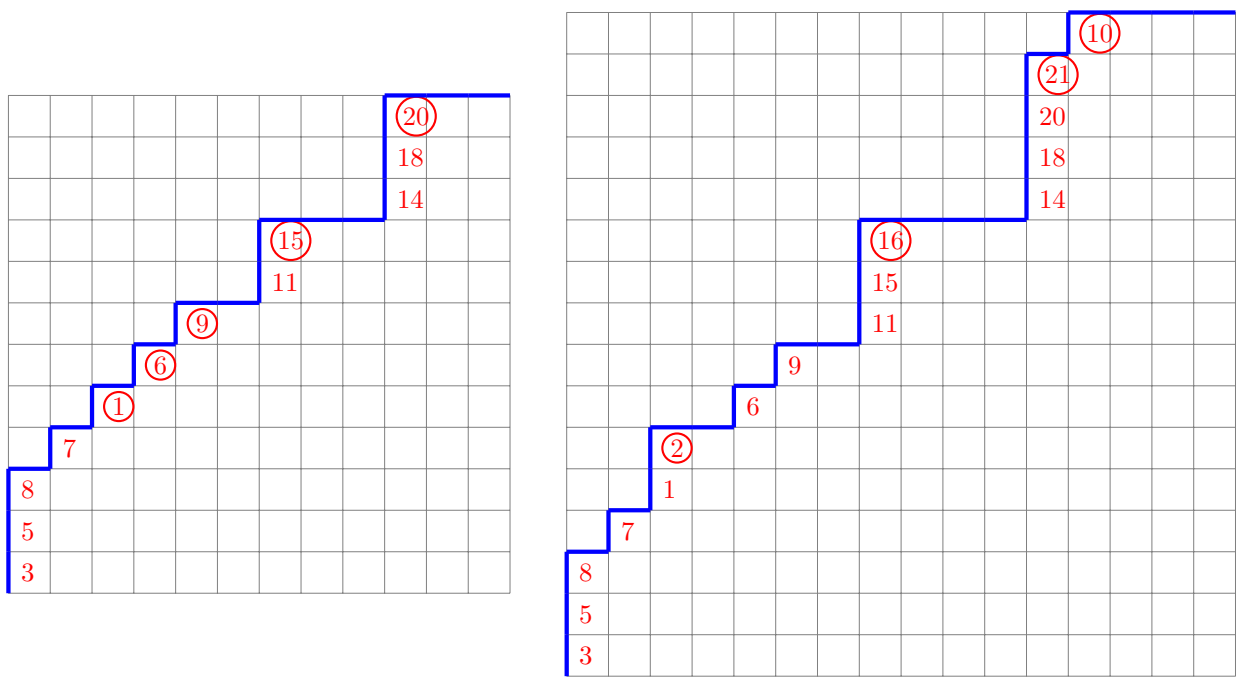

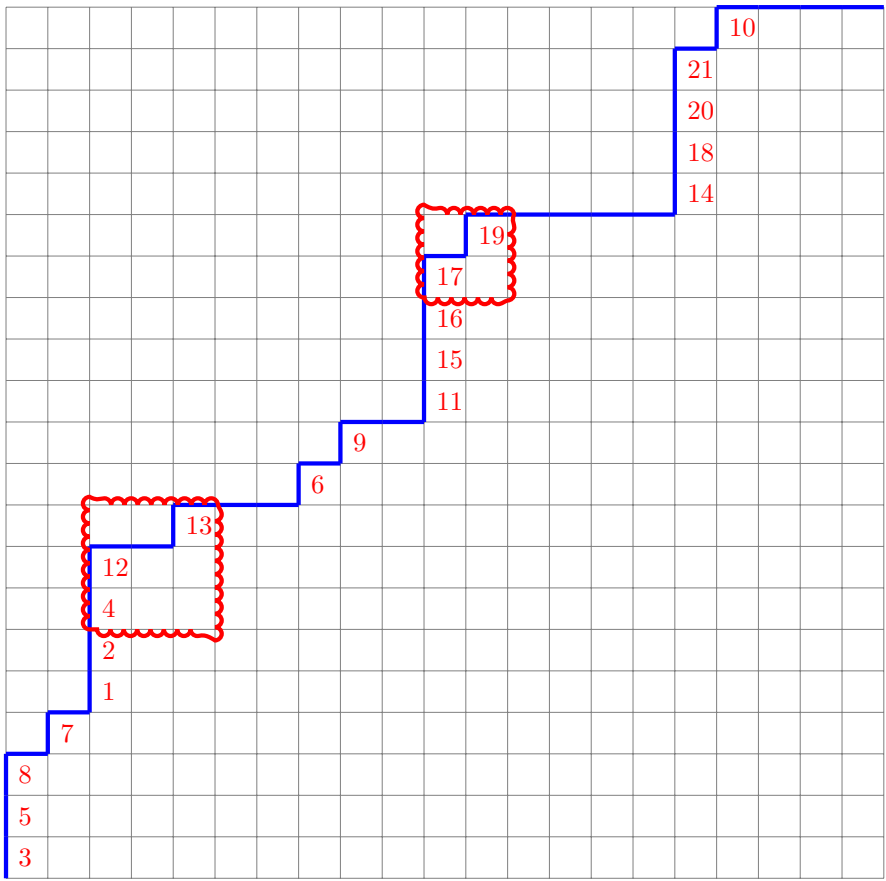
\begin{figure}
  \centering
  \begin{tikzpicture}[node distance=.4cm]
    \def\s{.55}
    \node  (phaseII) [below=of fD]{\begin{tikzpicture}[scale=\s,font=\footnotesize]
\begin{scope}
\draw[step=1cm,gray,very thin] (0,0) grid (21,21); 
\draw[ultra thick,blue] (0,0)--(0,1) node[red,midway,right] {$3$};
\draw[ultra thick,blue] (0,1)--(0,2) node[red,midway,right] {$5$};
\draw[ultra thick,blue] (0,2)--(0,3) node[red,midway,right] {$8$};
\draw[ultra thick,blue] (0,3)--(1,3);

\draw[ultra thick,blue] (1,3)--(1,4) node[red,midway,right] {$7$};
\draw[ultra thick,blue] (1,4)--(2,4);
\draw[ultra thick,blue] (2,4)--(2,5) node[red,midway,right] {$1$};
\draw[ultra thick,blue] (2,5)--(2,6) node[red,midway,right] {$2$};
\draw[ultra thick,blue] (2,6)--(2,7) node[red,midway,right] {$4$};
\draw[ultra thick,blue] (2,7)--(2,8) node[red,midway,right] {$12$};
\draw[ultra thick,blue] (2,8)--(3,8);
\draw[ultra thick,blue] (3,8)--(4,8);
\draw[ultra thick,blue] (4,8)--(4,9) node[red,midway,right] {$13$};
\draw[ultra thick,blue] (4,9)--(5,9);
\draw[ultra thick,blue] (5,9)--(6,9);
\draw[ultra thick,blue] (6,9)--(7,9);
\draw[ultra thick,blue] (7,9)--(7,10) node[red,midway,right] {$6$};
\draw[ultra thick,blue] (7,10)--(8,10);
\draw[ultra thick,blue] (8,10)--(8,11) node[red,midway,right] {$9$};
\draw[ultra thick,blue] (8,11)--(9,11);
\draw[ultra thick,blue] (9,11)--(10,11);
\draw[ultra thick,blue] (10,11)--(10,12) node[red,midway,right] {$11$};
\draw[ultra thick,blue] (10,12)--(10,13) node[red,midway,right] {$15$};
\draw[ultra thick,blue] (10,13)--(10,14) node[red,midway,right] {$16$};
\draw[ultra thick,blue] (10,14)--(10,15) node[red,midway,right] {$17$};
\draw[ultra thick,blue] (10,15)--(11,15);
\draw[ultra thick,blue] (11,15)--(11,16) node[red,midway,right] {$19$};
\draw[ultra thick,blue] (11,16)--(12,16);
\draw[ultra thick,blue] (12,16)--(13,16);
\draw[ultra thick,blue] (13,16)--(14,16);
\draw[ultra thick,blue] (14,16)--(15,16);
\draw[ultra thick,blue] (15,16)--(16,16);
\draw[ultra thick,blue] (16,16)--(16,17) node[red,midway,right] {$14$};
\draw[ultra thick,blue] (16,17)--(16,18) node[red,midway,right] {$18$};
\draw[ultra thick,blue] (16,18)--(16,19) node[red,midway,right] {$20$};
\draw[ultra thick,blue] (16,19)--(16,20) node[red,midway,right] {$21$};
\draw[ultra thick,blue] (16,20)--(17,20);
\draw[ultra thick,blue] (17,20)--(17,21) node[red,midway,right] {$10$};
\draw[ultra thick,blue] (17,21)--(18,21);
\draw[ultra thick,blue] (18,21)--(19,21);
\draw[ultra thick,blue] (19,21)--(20,21);
\draw[ultra thick,blue] (20,21)--(21,21);
\draw[ultra thick, red, decorate, decoration={bumps}] (2,6)--(2,9)--(5,9)--(5,6)--cycle;
\draw[ultra thick, red, decorate, decoration={bumps}] (10,14)--(10,16)--(12,16)--(12,14)--cycle;
\end{scope}

    \end{tikzpicture}};

\end{tikzpicture}
\caption{The parking function $h=\mainmap(f,g,D)$, where $(f,g,D)$ is given in Figure~\ref{fig:mainmapA}. See also Example~\ref{ex:mainmap}.}
\label{fig:mainmapC}
\end{figure}
\begin{example}\label{ex:mainmap}
We show how the map $\mainmap$ works on the functions $f$ and $g$ and
the sequence $D$ given in Figure~\ref{fig:mainmapA}. In the first
phase, the elements of $D$ are special inserted on $f$'s highest
diagonal, which is diagonal $3$. They are circled in the diagram on
the left of Figure~\ref{fig:mainmapB}. The label $15$, for example,
cannot be inserted in a column which precedes the label $11$ from $f$,
or it would create a secondary dinv pair. In the parking function on
the right of Figure~\ref{fig:mainmapB}, the labels $2,16, 21$ and $10$ from $g$'s main diagonal
have been inserted (second phase). Finally, in Figure~\ref{fig:mainmapC}, we insert
the blocks from $g$.
  
\end{example}

\begin{proposition}
For $(f,g,D)\in\decomp(n)$, the parking function $h=\mainmap(f,g,D)$ has $\secdinv(h)=0$. 
\end{proposition}
\begin{proof}
  We refer to the construction of $\mainmap$ given in this section.

  By Claim~\ref{claim:specinsert}, the parking function $h_{s+t}$ has zero secondary dinv, since it was built from $f\in\PFZ{(\dom(f))}$ using special insertion  with $c$ and $d$ satisfying the conditions in  Claim~\ref{claim:specinsert}. We need only show that we did not create any secondary dinv with the insertion of a block in the third phase.

   Suppose $i<j$ and $i,j\in\dom(g)$. By our construction, $\diaga h i=\diaga g i+\maxdiag f +1$ and similarly for $j$. What's more, the $z_k$ and their corresponding blocks were added in increasing order of their columns in $g$. Therefore, $\diaga h i=\diaga h j-1$ if and only if $\diaga g i=\diaga g j-1$ and $\cola h i>\cola h j$ if and only if $\cola g i >\cola g j$. Since $(j,i)$ was not a secondary dinv pair in $g$, it cannot be one in $h$. No new secondary dinv pair was created by an interaction of a label from the second phase and one from third, or two labels from the third.

Suppose we have a label from the first phase and one from the third. Their diagonals differ by at least two, so they could not form a secondary dinv pair.  
  \end{proof}

 \section{Inversion of the map $\mainmap$}
 \label{sec:mapinv}
 In this section, we show that $\mainmap$ is invertible, simply by reversing all the steps. We will “remove” rows, columns, and labels from $h$, to construct $D$ and $g$, and what remains of $h$ will be the parking function $f$.

 Let $h\in\PFZ(n)$. We must find $(f,g,D)\in\decomp(n)$.
We use $\mcd$ to denote $\diaga h 1$. The first step is to identify labels $z_1,z_2,\ldots,z_t$. They are the elements of diagonal $\mcd+1$, ordered so that $\cola h {z_1}<\cola h {z_2}<\cdots<\cola h {z_t}$ and will make up the main diagonal of $g$. 
 
Next, we define blocks $B_i$ for $1\le i< t$ as the set of labels $x$
in $h$ such that $\diaga h x>\mcd+1$ and $\cola h {z_i}\le \cola h
x<\cola h {z_{i+1}}$. The block $B_t$ is the set of $x$ with $\diaga h
x>\mcd+1$ and $\cola h {z_t}\le \cola h x$, it should be noted that $z_{i} \not \in B_{i}$ as $diag_{h} (z_{i}) = \mcd +1$.
The height $\hta B i$ of the block
$B_i$ is the number of rows occupied in $h$ by elements of $B_i$, plus $1$ for $z_i$, so that $\rowa h{z_i}+\hta B i= \rowa h{z_{i+1}}$. Note that $\hta B i=\WD {B_i}+1$, where $\WD {B_i}$ was defined in Section~\ref{subsec:map}. There will be at least $\hta B i-1$ columns which either contain elements of $B_i$ or are empty strictly between the column of $z_i$ and the column of $z_{i+1}$.  Additionally, the domain of $g$ has $b=\hta B 1+\hta B 2+\cdots+\hta B t$ elements.

The parking function $g$ is constructed by first placing the blocks,
together with $z_1,z_2,\ldots,z_t$, in a $b\times b$ grid so that
$z_1,\ldots,z_t$ are on the main diagonal. Place $z_i$ on the main
diagonal in column $\hta B 1+\cdots+\hta B {i-1}+1$.  Next the elements in
$B_i$ are placed in the grid. They must retain their relative position
to $z_i$, and every row must have exactly one element in it. 

We remove from $h$ the rows containing elements from a block and columns $\cola h {z_{i}} +1, \cola h {z_{i}} +2,\ldots, \cola h {z_{i}} +\hta B i -1$, also removing the labels in these rows and columns. We remove all labels in the column of $z_i$ from higher rows.  We call the resulting, smaller parking function $h'$. Notice that $\diaga h x=\diaga {h'} x$ for all labels $x$ remaining in $h'$.

The parking function $h'$ may now have fewer rows, columns, and
labels. The labels $z_1,\ldots,z_t$ are still in $h'$ and are on the
highest diagonal. We want to erase the labels $z_1,\ldots,z_t$ from
$h'$, but in the order $z_t,\ldots,z_1$ and always removing the column
following each $z_i$'s column. We begin with $z_t$. If there is a
label $x$ such that $\rowa {h'} x=\rowa {h'} {z_t}+1$, then since $\diaga {h'}
{z_i}>\diaga {h'} x$, we have $\cola {h'} x -\cola {h'} {z_t}>1$, so we remove
the empty column $\cola {h'} {z_t}+1$ and row $\rowa {h'} {z_t}$, thereby
also erasing $z_t$. If no such $x$ exists, then since $\diaga {h'}
{z_t}>\diaga {h'} 1\ge 0$, there is an empty column after $\cola {h'} {z_i}$,
which we remove, as well as the row $z_t$ is in. In either case, we call the resulting parking function $h_t$. We repeat this
procedure with $z_{t-1}$ down to $z_1$, producing parking functions $h_{t-1},\ldots,h_1$. At the end, $h_1$ has $\hta B
1+\hta B 2+\cdots+\hta B t$ fewer rows and columns than the original
$h$ and all remaining labels are in the same diagonal as in the
original $h$; i.e. $\diaga {h_1} x=\diaga h x$, for all labels $x$ in $h_1$.

The sequence $D=(e_1=1,e_2,\ldots,e_s)$ is taken from diagonal $\mcd$ of $h$, which is the highest diagonal of $h_1$. The elements $e_1,e_2,\ldots,e_s$ are the labels in the highest diagonal whose column indices are at least the column index of the label $1$ in $h_1$ and they are ordered so that $\cola {h_1} 1 <\cola {h_1} {e_2}<\cdots<\cola {h_1} {e_s}$. The labels $e_1,\ldots,e_s$ are removed from $h_1$, resulting in $f$. This step reverses the first phase of the construction of $\mainmap$. Finally, we mention $f$ will be on a domain of size $n-(s+b)$.
It is straightforward to see that we have reversed the procedure which defines $\mainmap$ and that this inverse function $\mainmap^{-1}$ is defined on all of $\PFZ(n)$.

 \subsection{The map $\mainmap$ and Theorem~\ref{thm:enumPFZ}}
   \label{subsec:enum}
We can now directly see that $\pfz{n}$ satisfies \eqref{Recursion} and thereby prove Theorem~\ref{thm:enumPFZ}. Let $k\in[n]$ and $0\le\ell\le n-k$. There are $\binom{n-1}{k-1}\cdot(k-1)!=\frac{(n-1)!}{(n-k)!}$ sequences $(e_1,e_2,\ldots,e_k)$, where $e_1,e_2,\ldots,e_k$ are distinct elements of $[n]$ and $e_1=1$. There are $\binom{n-k}{\ell}$ subsets $A\subset[n]$ of size $\ell$ which are disjoint from $\{e_1,\ldots,e_k\}$ and $\pfz{\ell}$ is the number of parking functions with zero secondary dinv on $A$. Let $B$ be the complement in $[n]$ of $\{e_1,\ldots,e_k\}\cup A$; there are $\pfz{n-k-\ell}$ parking functions with zero secondary dinv on $B$. We sum over $k$ and $\ell$ to obtain \eqref{Recursion}.  
   
   \section{Acknowledgments} The authors thank the reviewers for their suggestions, which significantly improved the paper.

{\bf Conflict of interest.} The authors declare that there is no conflict of interest regarding
the publishing of this paper.
 
   \bibliographystyle{plain}
\bibliography{dinv.bib}
\end{document}